\documentclass[11pt, oneside]{amsart}
\usepackage{amsmath}
\usepackage{amsthm}
\usepackage{amssymb} 
\usepackage[mathscr]{euscript}
\usepackage{comment}
\excludecomment{toexclude} 
\usepackage{bm}
\usepackage{fancyhdr}
\usepackage{enumerate}
\usepackage{amscd}
\usepackage[all]{xy}
\usepackage{graphicx}
\usepackage{color}
\usepackage{hyperref}
\hypersetup{
    colorlinks,
    citecolor=black,
    filecolor=black,
    linkcolor=black,
    urlcolor=black
}
\usepackage{tikz}
\usepackage[all]{xy}
\usepackage{graphicx}
\usetikzlibrary{backgrounds}
\usepackage{slashed}
\usepackage{accents}

\renewcommand{\a}{\alpha}

\newcommand{\g}{\gamma}

\newcommand{\D}{\Delta}

\newcommand{\s}{\sigma}
\newcommand{\Si}{\Sigma}

\newcommand{\cM}{{\mathcal M}}

\newcommand{\cL}{{\mathcal L}}

\newcommand{\cD}{{\mathcal D}}

\newcommand{\cX}{{\mathcal X}}

\newcommand{\bR}{\mathbb R}

\newcommand{\bS}{\mathbb S}

\newcommand{\bE}{\mathbb E}

\newcommand{\be}{\begin{equation}}
\newcommand{\ee}{\end{equation}}
\newcommand{\bes}{\begin{equation*}}
\newcommand{\ees}{\end{equation*}}

\newcommand{\pa}{\partial}
\renewcommand{\to}{\rightarrow}
\newcommand{\<}{\langle}
\renewcommand{\>}{\rangle}
\newcommand{\w}{\widetilde}

\theoremstyle{plain}
\newtheorem{lemma}{Lemma}[section]
\newtheorem{proposition}[lemma]{Proposition}
\newtheorem{theorem}[lemma]{Theorem}
\newtheorem{Theorem}{Theorem}

\newtheorem{definition}[lemma]{Definition}

\numberwithin{equation}{section}

\theoremstyle{definition}
\newtheorem{remark}[lemma]{Remark}

\newtheorem{notation}[lemma]{Notation}

\DeclareMathOperator{\tr}{tr} 
\DeclareMathOperator{\Div}{div}

\DeclareMathOperator{\Ric}{Ric}

\DeclareMathOperator{\Ker}{Ker}

\DeclareMathOperator{\C}{\mathcal{C}}

\newcommand{\R}{\mathbb{R}}

\newcommand{\m}{k}
\newcommand{\ml}{{k\ell}}
\renewcommand{\sc}{\mathrm{sc}}
\renewcommand{\g}{\mathfrak{g}}

\newcommand{\gsc}{\mathfrak{g}_{\mathrm{sc}}}

\newcommand{\CC}[2]{\mathcal{C}^{#1,\alpha}_{#2}(M)}

\usepackage{multicol} 
\usepackage{cancel} 

\usepackage[tmargin=1in,bmargin=1in,lmargin=1.5in,rmargin=1.5in]{geometry}
\begin{document}

\title[]{Existence of static vacuum extensions for Bartnik boundary data near Schwarzschild spheres}
\author{Spyros Alexakis}
\address{Department of Mathematics, University of Toronto}
\email{alexakis@math.toronto.edu}
\author{Zhongshan An}
\address{Department of Mathematics, University of Michigan}
\email{zsan@umich.edu}
\author{Ahmed Ellithy}
\address{Department of Mathematics, University of Toronto}
\email{ahmed.ellithy@mail.utoronto.ca}
\author{Lan-Hsuan Huang}
\address{Department of Mathematics, University of Connecticut}
\email{lan-hsuan.huang@uconn.edu}

\thanks{Part of the work was completed while ZA and LH were in residence at the Simons Laufer Mathematical Sciences Institute in Berkeley, California, during the Fall 2024 semester, supported by the NSF DMS-1928930. In addition, LH was partially supported by NSF DMS-2304966.}

\begin{abstract}
We obtain existence and local uniqueness of asymptotically flat, static vacuum extensions for Bartnik data on a sphere near the data of a sphere of symmetry in a Schwarzschild manifold.

 \end{abstract}
\maketitle 
\section{Introduction}
Let $M$ be an $n$-dimensional smooth manifold where $n\ge 3$. A Riemannian manifold $(M, \mathfrak g)$ is called \emph{static vacuum} if there exists a scalar function $f$ on $M$ that is not identically zero  such that 
\begin{equation} \label{orig}
    \begin{cases}
        &f\Ric_{\mathfrak{g}} = \text{Hess}_{\mathfrak{g}} f\\
        &\Delta_{\mathfrak{g}} f = 0
    \end{cases}
   \ \  \mbox{ in }M.
\end{equation}
The pair $(\mathfrak g, f)$ is referred to as a static vacuum pair, and $f$ is called a static potential. In \cite{Bartnik:1989}, Bartnik proposed the notion of Bartnik quasi-local mass, defining the mass of a compact Rimannian manifold with nonempty boundary $(\Omega, \mathfrak g_0)$ as the infimum of the ADM masses over all its admissible asymptotically flat extensions. In particular, those extensions must have the same Bartnik boundary data—specifically, the induced metric and mean curvature—on the boundary as $(\Omega, \mathfrak g_0)$. It was conjectured by Bartnik and later proved by Covino \cite{Corvino:2000} and Anderson-Jauregui \cite{Anderson-Jauregui:2019} that an extension that realizes the infimum ADM mass must be static vacuum with a static potential $f\to 1$ at infinity. 

The above observations naturally lead to another conjecture of Bartnik on the existence and uniqueness of asymptotically flat static vacuum extensions for prescribed Bartnik data, see, for example~\cite[Conjecture 7]{Bartnik:2002}. More precisely, on an asymptotically flat manifold $M$ with nonempty boundary $\Si:=\partial M$, given a Riemannian metric $\tau$ and a scalar function $\phi$ on $\Si$, the question is whether there exists a unique asymptotically flat static vacuum pair $(\mathfrak g, f)$ on $M$  which satisfies the Bartnik boundary condition
\begin{equation}\label{bdry}
 \mathfrak{g}^\intercal = \tau \quad \mbox{ and } \quad  H_{\mathfrak{g}} = \phi \quad  \quad \mbox{ on } \Si.
\end{equation}
Here $\mathfrak g^\intercal$ and $H_{\mathfrak g}$ denote the induced metric and mean curvature of the boundary $\Sigma\subset (M,\mathfrak g)$, respectively. Recall that the pair is referred to as the Bartnik boundary data.  (Our convention of the mean curvature is the tangential divergence of the unit normal pointing to infinity.) 

The simplest static vacuum manifold is the Euclidean space $(\mathbb R^n, g_{\bE})$. Consider the exterior region of the unit ball in  3-dimensional Euclidean space, $M= \mathbb R^3 \setminus B^3$ with boundary $\Sigma = \mathbb S^2$. (Throughout the paper, we reserve the notation $\mathbb S^2$ for the unit sphere and $S^2$ for a topologically sphere.) The Bartnik boundary data of $(\mathbb R^3 \setminus B^3, g_{\mathbb E})$ consists of the standard round metric on the unit sphere $\gamma_{\bS^2}$ and constant mean curvature $H_{g_\bE}=2$. In \cite{Miao:2003}, Miao proves that for arbitrary data $(\tau,\phi)$ on $\bS^2$ which is reflection-symmetric and close to  $(\gamma_{\mathbb S^2}, 2)$, there exists a solution to the boundary value problem \eqref{orig} and \eqref{bdry}. Without the assumption of reflectional symmetry, Anderson \cite{Anderson:2015static} shows that the kernel of the linearized problem at $(\mathbb R^3 \setminus B^3, g_{\mathbb E})$ is ``trivial''  in the sense \eqref{eq:kernel} below. In \cite{An-Huang:2022} An and Huang prove the existence and local uniqueness of the boundary value problem for data  close to the Bartnik data of $(\bR^n\setminus\Omega, g_{\bE})$ where $\Omega$ is any bounded star-shaped domain with smooth boundary in $\mathbb R^n$. In~\cite{An-Huang:2024}, they extend the results for a generic static vacuum manifold, including generic exterior regions of Schwarzschild manifolds. 

For each $m\in \mathbb R$, we denote $m_0:=\max\{ 0, m\} \ge 0$. The family of 3-dimensional Schwarzschild metrics  and the corresponding static potentials on $(2m_0, \infty)\times S^2$ are given by
\begin{equation*}
    \mathfrak{g}_{\sc}= \left(1-\tfrac{2m}{r}\right)^{-1}dr^2 + r^2 \gamma_{\bS^2} \quad \mbox{ and } \quad   f_{\sc} = \sqrt{1-\tfrac{2m}{r}}
\end{equation*}
where $r\in (2m_0, \infty)$ is the radial coordinate.  (For $m<0$,  $r=0$ is a singularity, which the above definition excludes.) 
Let $r_0> 2m_0$, and consider $M=(r_0, \infty)\times S^2$, which we may also denote by $\mathbb R^3\setminus B_{r_0}$. The Bartnik data of the static vacuum manifold $(\bR^3\setminus B_{r_0},\gsc)$  on the boundary is given by 
\begin{align*}
\gamma_{r_0} := \gsc^\intercal=r_0^2\gamma_{\bS^2} \quad \mbox{ and } \quad H_{r_0} := H_{\gsc}=\tfrac{2}{r_0}\left(1-\tfrac{2m}{r_0}\right)^{\frac{1}{2}}.
\end{align*}
The work of \cite{An-Huang:2024} implies that for a generic value of $r_0$ in $(2m_0, \infty)$, there exist a static vacuum extension for Bartnik data close to the one mentioned above. The main purpose of the paper is to extend the result for \emph{all} $r_0\in(2m_0,\infty)$, using the implicit function theorem  established in~\cite[Theorem 3]{An-Huang:2024}.  However, we emphasize that extending the existence result from a generic $r_0$ to all $r_0$ is considerably more challenging and requires new insights into the geometry of Schwarzschild manifolds, which we explore in this paper. We note that a separate, self-contained proof is provided by the third author \cite{Ellithy}.

\begin{Theorem}\label{th:existence}
Let $m_0:=\max\{ 0, m\}$, $\alpha\in (0, 1)$, and $q\in (\frac{1}{2}, 1)$. For each $r_0\in (2 m_0, \infty)$, there exist positive constants $\epsilon_0,C$ such that for each $\epsilon\in (0, \epsilon_0)$, if $(\tau,\phi)$ satisfies $\|(\tau,\phi)-(\gamma_{r_0},H_{r_0})\|_{\C^{2,\a}(S^2)\times \C^{1,\alpha}(S^2)}<\epsilon$, then there exists an asymptotically flat,  static vacuum pair $(\g, f)$ with $\|(\g,f)-(\mathfrak{g}_{\sc},f_{\sc})\|_{\C^{2,\a}_{-q}(\mathbb R^3\setminus B_{r_0})}<C\epsilon$ solving the boundary condition $(\g^\intercal, H_{\g})=(\tau, \phi)$ on $S^2$. Furthermore, the solution $(\g, f)$ is geometrically unique in a neighborhood of $(\mathfrak{g}_{\sc},f_{\sc})$ in $\C^{2,\alpha}_{-q}(\mathbb R^3\setminus B_{r_0})$ and  depend smoothly  on $(\tau, \phi)$ in a suitable gauge.  
\end{Theorem}

See Section~\ref{se:conformal} for the notation of weighted H\"older spaces. To explain the proof, we consider the map for an asymptotically flat pair $(\g, f)$ associated with our boundary value problem:
\[
    (\g, f) \mapsto  T(\g, f) = \begin{array}{l}  
 \left\{ \begin{array}{l} f\Ric_{\g} -  \text{Hess}_{\g } f \\ 
                                  \Delta_{\g} f
	\end{array} \right. \quad \mbox{ in } \bR^3\setminus B_{r_0}\\
\left\{ \begin{array}{l} \g^\intercal \\ H_{\g} 
	\end{array} \right. \quad \mbox{ on } \partial B_{r_0}
	\end{array} .
\]
Theorem~\ref{th:existence} reduces to showing that that $T(g, f) = (0, 0, \tau, \phi)$. The linearization of the above map at $  (\g, f)=(\g_{\sc}, f_{\sc})$  is given by
\begin{align}\label{eq:linear}
\begin{split}
	L(\tilde \g,\tilde f) = \begin{array}{l}  
 \left\{ \begin{array}{l} f_{\sc}\Ric'_{\gsc} (\tilde \g) +\tilde f\Ric_{\gsc}- \text{Hess}_{\gsc} \tilde f-\text{Hess}'_{\gsc}(\tilde \g) f_{\sc} \\ 
                                  \Delta_{\gsc} \tilde f+ \Delta'_{\gsc}(\tilde \g)f_{\sc}
	\end{array} \right. \quad \mbox{ in } \bR^3\setminus B_{r_0}\\
\left\{ \begin{array}{l} \tilde \g^\intercal \\ H'_{\gsc} (\tilde \g)
	\end{array} \right. \quad \mbox{ on } \partial B_{r_0}
	\end{array}.
\end{split}
\end{align}
where $(\tilde \g, \tilde f)$ denotes the infinitesimal deformation of $(\g, f)$. (Technically speaking, we should have labeled  $L$ by $L_{r_0}$, but since we will fix $r_0$ in the argument,  it should not cause confusion.) Here and throughout the paper, we use a prime to denote the linearization of a geometric quantity at $\gsc$; for example, $\Ric'_{\gsc} (\tilde \g)=\left.\tfrac{d}{dt}\right|_{t=0}\Ric_{\gsc+t\tilde \g}$.

    Theorem 3 in \cite{An-Huang:2024} states that if the kernel of $L$ is ``trivial'' in the sense that  it only consists of infinitesimal deformations generated by  diffeomorphisms that fix the boundary and preserve the asymptotically flat structure, then Theorem~\ref{th:existence} holds. More precisely, the kernel of $L$ is ``trivial'' if    
\begin{align} \label{eq:kernel}
    \Ker L =\big \{ \big( \cL_X \gsc, X(f_{\sc}) \big)\big\}
\end{align}
where $X$ is a vector field vanishing on the boundary $\partial B_{r_0}$ and $X$ asymptotic to $Z$ at infinity for some Euclidean Killing vector field $Z$. See the precise definition of the space  of $X$ below in \eqref{eq:vector}. Therefore, our main task to Theorem~\ref{th:existence} is to establish the following theorem. 


\begin{Theorem}\label{Main1}
Let $m_0:=\max\{ 0, m\}$. For all $r_0\in (2 m_0, \infty)$,  the linearized operator at $(\bR^3\setminus B_{r_0}, \mathfrak{g}_{\sc}, f_{\sc})$ has a trivial kernel as described above by \eqref{eq:kernel}.
\end{Theorem}

Our approach to  Theorem \ref{Main1} involves conformally transforming $(\g, f)$ to $(g, u) := (f^2 \g, \ln f)$, with $g$ in geodesic coordinates, so that the conformally transformed static equation for $(g, u)$ in three dimensions resembles the Ricci-flat equation. Taking advantage of the fact that a Schwarzschild manifold has a global foliation by umbilic, constant mean curvature spheres, we find a \emph{global} geodesic gauge for infinitesimal deformations. The linearization of the conformally transformed equations then yields \emph{structure equations} for the infinitesimal deformation $(\tilde g, \tilde u)$, evolving along the radial direction. Then to show that the kernel of the linearized operator is trivial as in \eqref{eq:kernel}, it suffices to prove \(\tilde{u} = 0\) in the global geodesic gauge. See Section~\ref{se:conformal} for details.

We show $\tilde u=0$ in Section~\ref{se:vanish}. We decompose $\tilde u$  into spherical harmonics and derive second-order differential equations for their coefficients from the structure equations with initial conditions given by the Bartnik boundary conditions. By establishing an ODE comparison result and examining the underlying properties of the ODE, we show that solutions decaying at infinity must vanish identically.

\section{Conformal formulation and structure equations} \label{se:conformal} 

For each $m$, we let $m_0:=\max\{ 0, m\}$. We denote by $(\mathfrak{g_{\sc}}, f_{\sc})$ the Schwarzschild static vacuum pair with mass $m$ defined on  $(2m_0, \infty)\times S^2$  as follows:
\begin{equation*}
    \gsc = \left(1-\tfrac{2m}{r}\right)^{-1}dr^2 + r^2 \gamma_{\bS^2} \quad \mbox{ and } \quad   f_{\sc} = \sqrt{1-\tfrac{2m}{r}}
\end{equation*}
where $r\in (2m_0, \infty)$ is the radial coordinate. Let $r_0 > 2m_0$. Denote $M:=  \bR^3\setminus B_{r_0}$ and  the boundary $\Si\cong S^2$.  The manifold $M$ is naturally equipped with the standard Cartesian coordinates $\{x^1,x^2,x^3\}$, and using this we can define the weighted H\"older space $\mathcal{C}^{k,\alpha}_{-q}(M)$, where $k = 0, 1, 2, \dots$, $\alpha \in (0,1)$, and $q \in \R$. The space $\C^\infty_{-q}(M)$ denotes the intersection of $\C^{k,\alpha}_{-q}(M)$ for all $k$.  We refer to the definition of weighted H\"older spaces in, for example, \cite[Section 2.1]{An-Huang:2022}. For $f \in \CC{k}{-q}$, we will sometimes write $f = O^{k,\alpha}(|x|^{-q})$ to emphasize the asymptotic rate. We slightly abuse notation and say that a tensor is in $\CC{k}{-q}$ if all of its component functions in Cartesian coordinates are in $\CC{k}{-q}$.
Throughout the paper we will assume that $q\in (\frac{1}{2}, 1)$. 

A Riemannian metric $\mathfrak{g}$ on $M$ is called asymptotically flat if $\mathfrak{g} - g_{\mathbb{E}} \in \CC{2}{-q}$, where $g_{\mathbb{E}}$ the Euclidean metric. If $\mathfrak{g}$ is asymptotically flat and $f$ is a positive scalar function such that $f-1 \in \CC{2}{-q}$, we refer to $(\mathfrak{g},f)$ as an \emph{asymptotically flat pair}. We denote by $\mathcal{M}(M)$ the space of asymptotically flat pairs  on $M$. Obviously both static vacuum pairs $(g_{\mathbb{E}}, 1)$ and $(\gsc, f_{\sc})$ are in $\mathcal{M}(M)$. We use $\mathfrak{g}^\intercal$ and $H_{\mathfrak{g}}$ to denote the induced metric and mean curvature of the boundary $\Sigma$, respectively. The mean curvature is defined by $H_{\mathfrak{g}} := \Div \nu_{\mathfrak{g}}$ where $\nu_{\mathfrak{g}}$ is the unit normal vector field pointing toward infinity.  
 

We recall the defintions of $T$  and specify the function spaces: 
\begin{align*}
&T: \mathcal M(M) \to \C^{0,\alpha}(M)\times \C^{2,\alpha}(\Sigma)\times \C^{1,\alpha}(\Sigma)\\
    &(\g, f) \mapsto  T(\g, f)=\begin{array}{l}  
 \left\{ \begin{array}{l} f\Ric_{\g} -  \text{Hess}_{\g } f \\ 
                                  \Delta_{\g} f
	\end{array} \right. \quad \mbox{ in } M\\
\left\{ \begin{array}{l} \g^\intercal \\ H_{\g} 
	\end{array} \right. \quad \mbox{ on } \Si
	\end{array} .
\end{align*}
The \emph{infinitesimal deformation} $(\tilde \g, \tilde f)$ belongs to the tangent space $T_{(\gsc, f_{\sc})} \mathcal M(M)$, which we identity as $\C^{2,\alpha}(M)$. Then the functions spaces for the linearized map $L$ at $(\gsc, f_{\sc})$, as defined by \eqref{eq:linear}, are given by $L: \C^{2,\alpha}(M)\to \C^{0,\alpha}(M)\times \C^{2,\alpha}(\Sigma)\times \C^{1,\alpha}(\Sigma)$. To prove Theorem~\ref{Main1}, our main task is to analyze the kernel space 
\[
\Ker L :=\Big\{ (\tilde \g, \tilde f) \in \C^{2,\alpha}_{-q} (M): L(\tilde \g, \tilde f)=0\Big\}.
\]

Following \cite[Definition 2.1]{An-Huang:2024}, we define the subgroup of $\C^{3,\alpha}_{\mathrm{loc}}$ diffeomorphisms of $M$ that fix the boundary and the structure at infinity as 
\begin{align*}
    \mathcal{D}(M) = \Big\{ \psi \in \CC{3}{\mathrm{loc}}  \, :\, &\psi |_{\Sigma} = \mathrm{Id}_{\Sigma}, \text{ and } \psi(x) - Ox= O^{3,\alpha}(|x|^{1-q})\\
    & \text{ for some constant matrix $O \in SO(3)$} \Big\}.
\end{align*}
Let $\mathcal{X}(M)$ be the tangent space of $\mathcal{D}(M)$ at the identity map:
    \begin{align}\label{eq:vector}
    \begin{split}
        \mathcal{X}(M) = \Big\{ X \in \CC{3}{\mathrm{loc}} : & \text{ $X=0$ on $\Sigma$ and $X-Z = O^{3,\alpha}(|x|^{1-q})$ }\\ &\text{ for some Euclidean  rotational Killing vector field $Z$} \Big\}.
    \end{split}
    \end{align}
Since the equation $T(\g, f) = (0, 0, \tau, \phi)$ are invariant under diffeomorphisms in $\cD(M)$; namely, $T(\psi^* \g, \psi^*f)= T(\g, f)$ for $\psi\in \cD(M)$. By taking the $t$-derivative of $T(\psi_t^* \g, \psi_t^*f)= T(\g, f)$ for any smooth path $\psi_t\in \cD(M)$,  it follows that  
\[
L\big(\mathcal{L}_{X} \gsc, X (f_{\sc}) \big)=0 \quad \mbox{for any $X\in\cX(M)$}. 
\]

We  note a regularity lemma for $\Ker L$.
\begin{lemma}[{Cf. Corollary 3.11 of \cite{An-Huang:2024}}]\label{reg}
If $(\tilde \g, \tilde f) \in \Ker L$, then there exists $X\in\cX(M)$ such that $(\tilde \g, \tilde f)+ \big(\mathcal{L}_{X} \gsc, X(f_{\sc}) \big)\in\C^{\infty}_{-q}(M)$. 
\end{lemma}
\begin{remark}
By the above lemma, to prove Theorem~\ref{Main1}, it suffices to consider kernel elements of $L$ that are $\C^\infty_{-q}(M)$.     
\end{remark}



\subsection{Conformal transformation}
Given $(\mathfrak{g}, f)\in\cM(M)$, we define 
\be\label{conformal}
g= f^2 \mathfrak{g} \quad \mbox{ and }\quad  \ u= \ln f.
\ee
Let $\cM_c(M)$ denote the space of pairs $(g,u)$ consisting of asymptotically flat metrics $g$ and scalar functions $u\in \C^{2,\a}_{-q}(M)$. It is straightforward to check the above conformal transformation above gives a one-to-one correspondence between $\cM(M)$ and $\cM_c(M)$. There is also a one-to-one correspondence between deformations $(\tilde \g,\tilde f)\in T_{(\gsc, f_{\sc})}\mathcal M(M)$ and deformations $(\tilde g,\tilde u)\in T_{(g_{\sc},u_{\sc})}\cM_c(M)$ given by 
\be\label{conformaldef}
\tilde g=f^2_{\sc}\tilde \g+2 \tilde f f_{\sc}\gsc \quad \mbox{ and } \quad  \tilde u=f^{-1}_{\sc} \tilde f.
\ee

The pair $(\mathfrak g, f)$ being static vacuum is equivalent to the pair $(g,u)$ solving the following system 
\begin{equation}\label{conformalorig}
\begin{cases}
    \Ric_g - 2du\otimes du=0\\
    \Delta_g u=0
    \end{cases}
    \ \ \mbox{ in }M.
\end{equation}
The Bartnik boundary data of $(M,\mathfrak{g})$ transforms as follows:
\be\label{conformalbdry}
\mathfrak{g}^\intercal
=e^{-2u}g^\intercal\quad \mbox{ and } \quad  \ H_{\mathfrak g}=e^u \big( H_g-2\nu_g(u)\big) \quad \mbox{ on }\Si.
\ee

Let $(g_{\sc},u_{\sc})$ denote the pair obtained by conformal transformation of the Schwarzschild pair $(\gsc, f_{\sc})$. Explicitly, 
\begin{align}\label{eq:Sch}
   g_{\sc} = dr^2+(r^2-2mr)\g_{S^2} \quad \mbox{ and } \quad u_{\sc} = \tfrac{1}{2} \ln \left(1 - \tfrac{2m}{r} \right). 
\end{align}

Let $r_0 > 2m_0 = \max \{0, 2m\}$ and $M:= \mathbb R^3\setminus B_{r_0}$. We consider the associated boundary value map $T_c$ for  $(g, u)$ 
\begin{align*}
&T_c: \mathcal M_c(M) \to \C^{0,\alpha}(M)\times \C^{2,\alpha}(\Sigma)\times \C^{1,\alpha}(\Sigma)\\
    &(g, u) \mapsto  T_c(g, u)= \begin{array}{l}  
 \left\{ \begin{array}{l} \Ric_g - 2du\otimes du \\ 
                                  \Delta_{g} u
	\end{array} \right. \quad \mbox{ in } M\\
\left\{ \begin{array}{l} e^{-2u} g^\intercal \\ e^u (H_g - 2\nu_g (u))
	\end{array} \right. \quad \mbox{ on } \Si
	\end{array} .
\end{align*}
The linearization of $T_c$ at $(g_{\sc}, u_{sc})$ is given by 
\begin{align}\label{conformalLinear}
\begin{split}
  L_c (\tilde g, \tilde u)= \begin{array}{l} \left\{\begin{array}{l}  \Ric'_{g_{\sc}}(\tilde g) - 2d\tilde u\otimes du_{\sc}-2du_{\sc}\otimes d\tilde u\\
    \Delta_{g_{\sc}} \tilde u+\Delta'_{g_{\sc}}(\tilde g) u_{\sc}
\end{array} \right. \quad \mbox{ in }M\\
\left\{ \begin{array}{l} 
   e^{-2u_{\sc}}  \big(\tilde g^\intercal-2\tilde u g_{\sc}^\intercal \big)\\
    e^{u_{\sc}} \Big( \tilde u \big( H_{g_{\sc}}-2\partial_r u_{\sc}\big)+ H'_{g_{\sc}}(\tilde g)-2\partial_r \tilde u-2\nu'_{g_{\sc}}(\tilde g) (u_{\sc}) \Big)
\end{array}\right. \quad \mbox{ on }\Si
\end{array}.
\end{split}
\end{align}


\begin{lemma}\label{hv-gu}
Let $(\tilde \g,\tilde f)$ and $(\tilde g,\tilde u)$ satisfy the relation \eqref{conformaldef}. Then 
\begin{itemize}
\item $L(\tilde \g,\tilde f) =0$ if and only if $L_c(\tilde g,\tilde u)=0$.
    \item $(\tilde \g,\tilde f) = (\cL_X \gsc, X(f_{\sc}))$ for some $X\in\cX(M)$ if and only if $(\tilde g, \tilde u)= (\cL_X g_{\sc},X(u_{\sc}) )$ for the same vector field $X$. 
\end{itemize}
\end{lemma}
\begin{proof}
The first item is obvious. For the second item, we assume $(\tilde \g,\tilde f) = (\cL_X \gsc, X(f_{\sc}))$. Then by \eqref{conformaldef} we have
\begin{equation*}\begin{split}
\tilde g&=f^2_{\sc}\cL_X\gsc+2X(f_{\sc}) f_{\sc}\gsc=f^2_{\sc}\cL_X(f^{-2}g_{\sc})+2X(f_{\sc}) f^{-1}_{\sc}g_{\sc}=\cL_X g_{\sc} \\
\tilde u&=f^{-1}_{\sc} X(f_{\sc})=X(\ln f_{\sc})=X(u_{\sc}).
\end{split}\end{equation*}
The reverse can be obtained similarly by using the inverse transformation of \eqref{conformaldef} given by $\tilde \g=f_{\sc}^{-2}\tilde g-2\tilde u \gsc$ and $\tilde f=f_{\sc}\tilde u$. 
\end{proof}

\subsection{Global  geodesic gauge and structure equations} 
From now on, we  consider the boundary value problem at the conformally changed background pair $(g_{\sc}, u_{\sc})$ given by 
\begin{align}\label{eq:Sch}
   g_{\sc} = dr^2+(r^2-2mr)\g_{S^2} \quad \mbox{ and } \quad u_{\sc} = \tfrac{1}{2} \ln \left(1 - \tfrac{2m}{r} \right). 
\end{align}

Let $(\tilde g, \tilde u)$ be an infinitesimal deformation of $(g_{\sc}, u_{\sc})$, namely, $(\tilde g, \tilde u)\in T_{(g_{\sc}, u_{\sc})} \mathcal M_c (M)$. Conventionally, $\tilde g$ is said to satisfy the geodesic gauge if $\tilde g(\partial_r, \cdot)=0$, and it can be always obtained (by adding an infinitesimal deformation generated by a diffeomorphism) in a collar neighborhood of the boundary for a general manifold. See, for example, \cite[Lemma 2.8]{An-Huang:2024}. The advantage of using a geodesic gauge is that the information $\tilde g$ is fully encoded in the tangential components on constant $r$ level sets. For Schwarzschild manifolds, we can actually obtain the geodesic gauge on the entire $M$,  using the special geometric properties of $(M, g_{\sc})$.

\begin{definition}
A symmetric $(0,2)$-tensor $\tilde g$  defined on $(M, g_{\sc})$ is said to satisfy the \emph{global geodesic gauge} if $\tilde g(\partial_r, \cdot)=0$ everywhere in $M$.
\end{definition}

\begin{proposition}\label{Ggauge}
Let $\tilde g\in \C^\infty_{-q}(M)$ be a symmetric $(0,2)$-tensor on $(M, g_{\sc})$. Then there exists a unique vector field $X\in \C^\infty_{1-q}(M)$ vanishing on the boundary such that $\tilde g+\cL_X g_{\sc}$ satisfies the global geodesic gauge. 
\end{proposition}
\begin{proof}
Note that $r$ is the coordinate along the unit speed geodesic for $g_{\sc}$ and  $\nabla_{\partial_r} \partial_r = 0$. Let $\{e_A\}_{A=1,2}$ be an orthonormal frame on a constant $r$ sphere, and then we parallel transport $e_{A}$ along the $r$ coordinate. It is straightforward to check that $\{e_A\}$ is tangent to all constant $r$ spheres. 
We can decompose a vector field $X $ on $M$ as $X=X^{\perp}\partial_r+X^\intercal$ where $X^{\perp}=g_{\sc}(X,\partial_r)$ is the normal component and $X^\intercal = \sum_{A} g_{\sc} (X, e_A) e_A$ is tangent to the constant $r$ spheres.  
For given $\tilde g$, we will choose $X^\perp$ and $X^\intercal$ suitably so that 
\[
    \tilde g (\partial_r, \cdot) + \cL_X g_{\sc}(\partial_r, \cdot) = 0.
\]

By direct computations, we have 
\begin{align*}
        \big(\cL_X g_{\sc}\big)(\partial_r, \partial_r ) &=2g_{\sc}(\nabla_{\partial_r}X,\partial_r)=2\partial_r X^{\perp}\\
        \big(\cL_X g_{\sc}\big)(\partial_r,e_A)&=g_{\sc}(\nabla_{\partial_r}X,e_A)+g_{\sc}(\nabla_{e_A}X,\partial_r)\\
&=\partial_r g_{\sc}(X^\intercal,e_A)-g_{\sc}(X,\nabla_{\partial_r}e_A)+e_A(X^\perp)-g_{\sc}(X,\nabla_{e_A}\partial_r)\\
&=\partial_r g_{\sc}(X^\intercal,e_A)+e_A(X^\perp)-\tfrac{1}{2}H_{\sc}g_{\sc}(X^\intercal,e_A)
\end{align*}
where in the last line we use $\nabla_{\partial r} e_A=0$ and the constant $r$ spheres are umbilic, and $H_{\sc}$ denotes the mean curvature of the constant $r$ spheres.  

We determine $X$ as follows. Let $X^\perp = 0$ on $\Sigma = \{ r = r_0\}$, and define $X^\perp (r, \theta)$
\[
        X^\perp (r, \theta) = -\frac{1}{2} \int_{r_0}^r \tilde g(\partial_r, \partial_r) (s, \theta)\, ds. 
\]
Once $X^\perp$ is determined on $M$, we define $X^\intercal$ so that each coefficient $g_{\sc}(X, e_A)$ is the unique solution solving the ODE in $r$ with the initial condition $g_{\sc}(X, e_A)= 0$ on $\Sigma$:
\[
\partial_r g_{\sc}(X^\intercal,e_A)-\tfrac{1}{2}H_{\sc}g_{\sc}(X^\intercal,e_A)=-\tilde g(\partial_r, e_A) -e_A(X^\perp).
\]
It is straightforward to verify that $X$ satisfies the desired regularity. 

\end{proof}

We remark that the above proposition  can be viewed as a linear version of the existence result of global radial foliations for metrics near $g_{\sc}$ obtained in  \cite[Theorem 5.0.1]{CK}.

To derive the key structure equations, we introduce the following notations.
\begin{notation}\label{notations}
We shall work with the constant $r$ foliation of $(M,g_{\sc})$ and use the following notations to describe the intrinsic and extrinsic geometry of the constant $r$ spheres, denoted by $S_r$. 

\begin{itemize}
\item We denote the induced metric $g_{\sc}^\intercal$ on $S_r$ by $\gamma_{\sc}= r(r-2m) \gamma_{\bS^2}$, where $\gamma_{\bS^2}$ is the metric of the unit sphere. We denote the infinitesimal deformation of $\gamma_{\sc}$ by $\tilde \gamma$; that is, $\tilde \gamma = \tilde g^\intercal$, where $\tilde g$ is the infinitesimal deformation of $g_{\sc}$. 
\item The scalar curvature of $S_r$ is $R_{\gamma_{\sc}} = \frac{2}{r(r-2m)}$ with the linearization at $\gamma_{\sc}$ denoted by $R'_{\gamma_{\sc}}(\tilde \gamma)$.
\item Let $\nu_{\sc} = \partial_r$ donote the unit normal to $S_r$. The second fundamental form on $S_r$ is given by $K_{\sc}=\frac{(r-m)}{r(r-2m)}\gamma_{\sc}$ and its traceless part is $\ring K_{\sc} = 0$. The mean curvature on $S_r$ is $H_{\sc} = \frac{2(r-m)}{r(r-2m)}$. We denote their linearizations at $g_{\sc}$ by $\tilde K, \widetilde {\ring K}$, and $\tilde H $, respectively. That is, 
 $\tilde H = H'_{g_{\sc}}(\tilde g)$ and similarly for other terms.

\item We add a slash on geometric operators with respect to $\gamma_{\sc}$ on $S_r$; for example, the covariant derivative $\slashed\nabla$, divergence $\slashed{\Div}$, Laplace operator $\slashed\Delta$, and trace  $\slashed{\tr}$.
\end{itemize}
\end{notation}
The following proposition describes the linearized conformal static vacuum equations in the radial foliation for deformations in the global geodesic gauge.
\begin{proposition}\label{structure}
Let $(\tilde g,\tilde u)$ solve the bulk equations equal to zero in \eqref{conformalLinear}; namely, 
\begin{align}
   \Ric'_{g_{\sc}}(\tilde g) - 2d\tilde u\otimes du_{\sc}-2du_{\sc}\otimes d\tilde u &= 0 \label{eq:conformal1}\\
    \Delta_{g_{\sc}} \tilde u+\Delta'_{g_{\sc}}(\tilde g) u_{\sc} &= 0.\label{eq:conformal2}
\end{align}
Suppose $\tilde g$ satisfies the global geodesic gauge (thus, $\tilde g = \tilde \gamma$). Then $(\tilde \gamma, \tilde u)$ satisfies the following structure equations on $M$:
\begin{subequations}\label{eq:DG}
\begin{align}
	\partial_r \tilde H+ H_{\sc} \tilde H  + 4(\partial_r u_{\sc})( \partial_r \tilde u)& = 0  \label{DG2}\\
	- 4(\partial_r u_{\sc}) (\partial_r \tilde u)&   = -H_{\sc}\tilde H +R_{\gamma_{\sc}}'(\tilde \gamma) \label{DG4}\\
	2(\partial_r u_{\sc}) \slashed d\tilde u &= \slashed \Div (\widetilde{\ring K})   -\frac{1}{2} \slashed d\tilde H \label{DG5}\\	
 \cL_{\pa_r}\w{\ring K}&=0\label{DG3}\\
	\pa^2_r \tilde u+H_{\sc}\pa_r \tilde u + \slashed \Delta \tilde u +  (\partial_r u_{\sc}) \tilde H  &= 0. \label{DG1}
\end{align}
\end{subequations}

If we further assume the zero boundary boundary conditions on $\Sigma$ from \eqref{conformalLinear}:
\begin{align*} 
   &\tilde g^\intercal-2\tilde u \gamma_{\sc} =0\\
   & \tilde u \big( H_{g_{\sc}}-2\partial_r u_{\sc}\big)+ \tilde H-2\partial_r \tilde u-2\nu'_{g_{\sc}}(\tilde g) (u_{\sc})=0,
\end{align*}
then $(\tilde \gamma, \tilde u)$ satisfies 
\begin{align}\label{DG6} 
&\tilde \gamma = 2 \tilde u \gamma_{\sc} \quad \mbox{ and } \quad \tilde H= 2\partial_r \tilde u- \tfrac{2}{r_0} \tilde u \quad \mbox{ on } \Sigma.
\end{align}
\end{proposition}
\begin{proof}
We recall the following basic equations for surfaces within a general 3-dimensional Riemannian manifold.
\begin{equation*}\begin{split}
\Ric_g(\nu,\nu)+\nu(H)+|\ring K|^2+\tfrac{1}{2}H^2&=0 \ \ \mbox{(Ricatti equation)}\\
R_g-2\Ric_g(\nu,\nu)&=|\ring K|^2-\tfrac{1}{2}H^2+R_{\gamma}\ \ \mbox{(double traced Gauss equation)}\\
\Ric_g(\nu)^\intercal&={\Div}_{\gamma}\ring K-\tfrac{1}{2}dH\ \ \mbox{(Codazzi equation)}\\
\Ric_g^\intercal-\tfrac{1}{2}(\tr_\gamma \Ric_g^\intercal)\gamma&=-\nabla_\nu \ring{K}-H \ring{K}.
\end{split}
\end{equation*}
where $\Ric_g(\nu)^\intercal$ and $\Ric^\intercal$ denote the tangential parts of $\Ric(\nu,\cdot)$ and $\Ric(\cdot, \cdot)$, respectively. To see that last equation above, we use the Gauss equation to obtain
\begin{align*}
	\Ric^\intercal &=  -\nabla_\nu K  + \tfrac{1}{2} R_\gamma \gamma - HK 
\end{align*}
and then subtracting the trace term.

We linearize the above equations at $g_{\sc}$.  We use the facts that  $\nu'(\tilde g)=0$ because $\tilde g$ satisfies the global geodesic gauge and that $\ring K_{\sc}=0$. We obtain
\begin{align}\label{eq:str-linear}
\begin{split}
\Ric'_{g_{\sc}}(\tilde g)(\partial r,\partial r)+\partial_r\tilde H+H_{\sc}\tilde H&=0 \\
R'_{g_{\sc}}(\tilde g)-2\Ric'_{g_{\sc}}(\tilde g)(\partial r,\partial r)&=-H_{\sc} \tilde H+R'_{\gamma_{\sc}}(\tilde \gamma)\\
\Ric'_{g_{\sc}}(\tilde g)(\pa_r)^\intercal&=\slashed{\Div}\widetilde{\ring K}-\tfrac{1}{2}\slashed d\tilde H\\
\Ric'_{g_{\sc}}(\tilde g)^\intercal-\tfrac{1}{2}(\tr_{\gamma_{\sc}} \Ric'_{g_{\sc}}(\tilde g)^\intercal)\gamma_{\sc}&=-\nabla_{\pa_r} \w{\ring K}-H_{\sc} \w{\ring K} = -\cL_{\pa_r}\w{\ring K}.
\end{split}
\end{align}
In the last identity, we use the conformal static equation $ \Ric_{g_{\sc}} =2 du_{\sc} \otimes du_{\sc}$ and hence $\Ric_{g_{\sc}}^\intercal=0$. We also use  that $S_r$ is umbilic to rewrite the right hand side as a Lie derivative.

Now, we derive the first four equations in \eqref{eq:DG}. Using that  $(\tilde g,\tilde u)$ is a conformal static vacuum deformation. Using that $(\tilde g, \tilde u)$ satisfies \eqref{eq:conformal1}, we have
\begin{align}\label{eq:static}
\begin{split}
\Ric'_{g_{\sc}}(\tilde g)(\pa_r,\pa_r)&=4(\pa_r u_{\sc}) (\pa_r\tilde u)\\
\Ric'_{g_{\sc}}(\tilde g)(\pa_r)^\intercal&=2(\pa_r u_{\sc}) \slashed d \tilde u\\
\Ric'_{g_{\sc}}(\tilde g)^\intercal&=0\\
R'_{g_{\sc}}(\tilde g)=\tr_{g_{\sc}}\Ric'_{g_{\sc}}-\tilde g^{-1} \Ric_{g_{\sc}}&=4(\pa_r u_{\sc})(\pa_r\tilde u),
\end{split}
\end{align}
where for the last identity we use  $\tr_{g_{\sc}}\Ric'_{g_{\sc}}=4\pa_r\tilde u\pa_r u_{\sc}$ and $\tilde g^{-1} \Ric_{g_{\sc}}=0$ because of the conformal static equation and the global geodesic gauge. Substitute equations from \eqref{eq:static} into \eqref{eq:str-linear} gives the first four equations of \eqref{eq:DG}. 


To obtain \eqref{DG1}, we simply write out \eqref{eq:conformal2} and use the formula for the linearization of the Laplace operator (see, for example \cite[Proposition 1.184]{BESSE}):
\begin{equation*}\begin{split}
0&=\Delta_{g_{\sc}}\tilde u-\<\mathrm{Hess} \, u_{\sc}, \tilde g\>-\<du_{\sc},{\Div}\tilde g\>+\tfrac{1}{2}\<du_{\sc},d(\tr\tilde g)\>\\
&=\Delta_{g_{\sc}}\tilde u-{\Div} (\tilde g(\nabla u_{\sc}))+\tfrac{1}{2}\<du_{\sc},d(\tr\tilde g)\>\\
&=\Delta_{g_{\sc}}\tilde u-{\Div} \big((\pa_r u)\tilde g(\pa_r)\big)+ \tfrac{1}{2}(\pa_r u_{\sc} ) \pa_r(\tr\tilde g)\\
&=\Delta_{g_{\sc}}\tilde u+(\pa_r u_{\sc})\tilde H
\end{split}\end{equation*}
where in the last line  we use the global geodesic gauge and the variation formula of mean curvature in the geodesic gauge (see, e.g. \cite[Appendix A]{An-Huang:2024}). Decomposing the Laplace into the tangential Laplace yields \eqref{DG1}.
\end{proof}

 To conclude this section, we observe that to show $(\tilde g,\tilde u)=0$, with $(\tilde g, \tilde u)$ satisfying $L_c(\tilde g, \tilde u) =0$, it suffices to show that $\tilde u = 0$.
\begin{lemma}\label{zeroG}
    Suppose that $(\tilde g, \tilde u)$ solves $L_c (\tilde g, \tilde u)=0$ and that $\tilde g$ satisfies the global geodesic gauge. If $\tilde u=0$, then $\tilde g=0$. 
\end{lemma}
\begin{proof}
Let $(\tilde g,\tilde u)$ solve \eqref{eq:DG} in $M$ and  \eqref{DG6} on $\Sigma$.  Since $\tilde g$ satisfies the global geodesic gauge,  we just need to show $\tilde \gamma=0$.

Setting $\tilde u=0$ in the boundary equations \eqref{DG6} yields 
\be\label{zeroB}
\tilde \gamma=0\quad \mbox{ and } \quad  \tilde H=0 \qquad \mbox{ on }\Si.
\ee
Restricting  \eqref{DG5} on $\Sigma$, we obtain that $\widetilde{\ring K}$ is divergence free, and thus it must be identically zero (as a TT tensor on $S^2$ must be zero). Thus, 
\be\label{zeroC}
\widetilde{\ring K}=0 \ \mbox{ on }\Si.
\ee
Then  \eqref{DG3} implies $\w{\ring K}=0$ on $M$. Setting $\tilde u=0$ in \eqref{DG2} also implie $\tilde H=0$ on $M$. Then we have  
\[
\w K=\w{\ring K}+\tfrac{1}{2}\tilde H\gamma_{\sc}+\tfrac{1}{2} H_{\sc}\tilde\gamma=\tfrac{1}{2} H_{\sc}\tilde\gamma \quad \mbox{ on }M.
\]
On the other hand, by the definition of the second fundamental form $\cL_{\nu} \gamma=2K$ we have $\cL_{\pa_r}\tilde\gamma=2\w K$. Therefore, $\tilde\gamma$ solves the ODE $\cL_{\pa_r}\tilde\gamma=H_{\sc}\tilde\gamma$  on $M$. It follows that $\tilde \gamma=0$ on $M$.


\end{proof}

\begin{remark}
An alternative proof of the above lemma is to observe that the boundary equations \eqref{zeroB} and \eqref{zeroC} imply that $(\tilde g,\tilde u)$ with $\tilde u=0$ satisfies the Cauchy boundary data in the sense of Theorem 5$^\prime$ in \cite{An-Huang:2024}.
\end{remark}

\section{Vanishing of the potential function deformation}\label{se:vanish}
The goal of this section is to prove the following theorem. 
\begin{theorem}\label{main}
Let $m_0= \max \{ 0, m\}$. For any $r_0>2m_0$, let $M:= \mathbb R^3\setminus B_{r_0}$ and $L_c$ be defined as in \eqref{conformalLinear}. Suppose $(\tilde g, \tilde u)  \in \C^\infty_{-q}(M)$ solves $L_c(\tilde g, \tilde u)=0$ and $\tilde g$ satisfies the global geodesic gauge. Then $\tilde u = 0$ on $M$, and hence $\tilde g=0$ on $M$ by Lemma~\ref{zeroG}. 
\end{theorem}
Assuming Theorem~\ref{main}, we prove Theorem \ref{Main1}. 
\begin{proof}[Proof of Theorem~\ref{Main1}]
Let $(\tilde \g, \tilde f) \in \Ker L$. By Lemma~\ref{reg}, we may without loss of generality assume $(\tilde \g, \tilde f) \in \C^\infty_{-q}(M)$. To show that \eqref{eq:kernel} holds, that is,  $(\tilde \g, \tilde f) = (\cL_X \gsc, X(f_{\sc}))$ for some $ X\in \cX(M)$,   it suffices to show that if $(\tilde g, \tilde u) \in \C^\infty_{-q}(M) $ solving $L_c(\tilde g, \tilde u)=0$, then   $(\tilde g, \tilde u) = (\cL_X g_{\sc}, X(u_{\sc}))$ for some $X\in \cX(M)$ by Lemma~\ref{hv-gu}. By Proposition~\ref{Ggauge}, we can without loss of generality assume $\tilde g$ satisfies the global geodesic gauge. Then the triviality of $(\tilde g, \tilde u)$ follows from Theorem~\ref{main}.

\end{proof}

For the following computations, we recall
\begin{align}\label{eq:explicit}
\begin{split}
    \gamma_{\sc} &= r(r-2m) \gamma_{\bS^2}, \quad \slashed \Delta = \tfrac{1}{r(r-2m)} \Delta_{\gamma_{\bS^2}}\\
 u_{\sc}&=\tfrac{1}{2} \ln \left( 1-\tfrac{2m}{r}\right), \quad \pa_r u_{\sc}=\tfrac{m}{r(r-2m)}, \quad \mbox{ and } \quad  H_{\sc}=\tfrac{2(r-m)}{r(r-2m)}.
\end{split}
\end{align}
To prove Theorem \ref{main}, we decouple the equation \eqref{DG1} for $\tilde u$ and $\tilde H$ to obtain a differential equation with initial condition for $\tilde u$ only. 
\begin{proposition}\label{fund}
Let $(\tilde \gamma, \tilde u)$ solve \eqref{DG2}, \eqref{DG1}, and the boundary condition~\eqref{DG6}. We write $\tilde u = \tilde u(r,\theta)$ where $(r, \theta)\in (r_0,\infty)\times S^2$. Then $\tilde u$  satisfy the following differential equation of $r$:
 \begin{align}\label{eq10}
 \begin{split}
&r(r-2m) \pa^2_r \tilde u (r,\theta)+2(r-m)\pa_r \tilde u (r,\theta)+ \Delta_{\gamma_{\bS^2}} \tilde u (r,\theta)-\frac{4m^2}{r(r-2m)} \tilde u (r,\theta)\\
&= -\frac{2m}{r(r-2m)} \Big((4m-r_0)\tilde u(r_0, \theta) + r_0 (r_0-2m) \partial_r \tilde u (r_0, \theta) \Big). 
\end{split}
\end{align}
If we further assume \eqref{DG4}, then  $\tilde u$ satisfies the following identity on $\Si$:
 \begin{align}
&2 r_0(r_0-2m) \partial_r \tilde u (r_0, \theta) + r_0 \Delta_{\gamma_{\mathbb S^2}} \tilde u (r_0,\theta)+ 2m   \tilde u (r_0, \theta) =0 \label{eq20}.
\end{align}
\end{proposition}

\begin{proof}
We reexpress  \eqref{DG2} and \eqref{DG1} by substituting the terms from \eqref{eq:explicit}:
\begin{align*}
    r(r-2m)\partial_r \tilde H+ {2(r-m)} \tilde H + {4m}\pa_r\tilde u=0\\
    \partial_r^2 \tilde u + \frac{2(r-m)}{r(r-2m)} \partial_r\tilde u + \frac{1}{r(r-2m)}\Delta_{\gamma_{\bS^2}} \tilde u+ \frac{m}{r(r-2m)} \tilde H =0
\end{align*}
Note that the first identity implies that  $r(r-2m) \tilde H + 4m\tilde u$ is constant in $r$. Therefore, 
\begin{align*}
  \tilde H (r, \theta) + \frac{4m}{r(r-2m)}\tilde u (r, \theta) &=\frac{r_0(r_0-2m) }{r(r-2m)}\tilde H (r_0, \theta) + \frac{4m}{r(r-2m)}\tilde u (r_0,\theta)\\
  &=\frac{2r_0(r_0-2m)}{r(r-2m)} \left( \frac{4m-r_0}{r_0(r_0-2m)} \tilde u(r_0, \theta) + \partial_r \tilde u(r_0, \theta)\right)
\end{align*}
where we substitute $\tilde H$ by  the boundary condition $\frac{1}{2} \tilde H(r_0, \theta) = -\frac{1}{r_0} \tilde u(r_0, \theta) + \partial_r \tilde u(r_0, \theta)$ from ~\eqref{DG6}. Substituting the identity for $\tilde H$ into the previous differential equation for $\tilde u$ gives
 \begin{align*}
\pa^2_r \tilde u (r,\theta)+&\frac{2(r-m)}{r(r-2m)}\pa_r \tilde u (r,\theta)+ \frac{1}{r(r-2m)}\Delta_{\gamma_{\bS^2}} \tilde u (r,\theta)-\frac{4m^2}{(r(r-2m))^2} \tilde u (r,\theta)\\
&= -\frac{2m r_0(r_0-2m) }{(r(r-2m))^2}\left(\frac{4m-r_0}{r_0(r_0-2m)}\tilde u(r_0, \theta) + \partial_r \tilde u (r_0, \theta) \right). 
\end{align*}
Multiplying the previous identity by $(r(r-2m))^2$ gives \eqref{eq10}.

The  boundary conditions \eqref{DG6} for $(\tilde g, \tilde u)$ do not naturally give a decoupled boundary condition for $\tilde u$. Therefore, we consider \eqref{DG4} on $\Sigma$
\begin{align*}
    R'_{\gamma_{\sc}} (\tilde \gamma) &= -4 (\partial_r u_{\sc} ) (\partial_r\tilde u) + H_{\sc}\tilde H \\
    &=    \frac{-4m}{r_0(r_0-2m)} \partial_r \tilde u (r_0,\theta) + \frac{2(r_0-m)}{r_0(r_0-2m)} \left( -\frac{2}{r_0} \tilde u(r_0, \theta) + 2\partial_r \tilde u(r_0, \theta)\right).
\end{align*}
On the other hand, $\tilde \gamma = 2\tilde u \gamma_{\sc}$, we have  
\begin{align*}
R'_{\gamma_{\sc}}(\tilde \gamma)&=-\slashed\D \big(\slashed\tr(2 \tilde u \gamma_{\sc})\big)+\slashed{\Div}\slashed{\Div}(2 \tilde u \gamma_{\sc})-\langle \Ric_{\g_{\sc}},2 \tilde u \gamma_{\sc}\>\\
&=-2\slashed\D \tilde u-\frac{4}{r_0(r_0-2m)} \tilde u.
\end{align*}
Equating the previous two identities give \eqref{eq20}.
\end{proof}

Let the integers $\ell$ and  $k$ range as $\ell=0,1,2,...$ and $\m=0,\pm 1,...,\pm \ell$. Let  $\{Y_\ml\}$ be the set of $L^2$-orthonormal spherical harmonics on the unit sphere $\bS^2$ satisfying
\[
  \D_{\gamma_{\bS^2}}Y_\ml=-\ell(\ell+1)Y_\ml .
\]
We write $\tilde u$ in the spherical harmonics decomposition:  
\begin{equation}\label{u sph harm}
 \tilde u(r, \theta) = \sum_{\ell=0}^{\infty} \sum_{\m=-\ell}^{\ell}  a_\ml(r) Y_\ml(\theta)
\end{equation}
where the coefficients $a_{\ml}(r)$ are  defined by
\begin{equation}
  a_\ml(r) = \int_{S^2} Y_\ml(x) \tilde u(r,\theta) d\s
\end{equation}
where $d\s$ is the area form of $\gamma_{\bS^2}$. Note that 
\[
    \frac{\partial^j}{\partial r^j} a_\ml(r) = \int_{S^2} Y_\ml(x) \frac{\partial^j}{\partial r^j} \tilde u(r,\theta) d\sigma.
\]
Since $\tilde u = O^{2,\alpha}(r^{-q})$, it follows that $a_\ml(r) = O^{2,\alpha}(r^{-q})$. 

The differential equation \eqref{eq10} and the boundary condition \eqref{eq20} implies the following differential equations for each coefficient $a_\ml(r)$, along with the initial value condition:
\begin{align} \label{ode-333}
\begin{cases} 
&\left(r(r-2m) a_\ml' \right)'=   \left( \frac{4m^2}{r(r-2m)} +\ell(\ell+1)\right)a_\ml(r) \\ 
&\qquad\qquad \qquad  \qquad  -\frac{2m}{r(r-2m)}  \Big((4m-r_0)  a_\ml(r_0)+ r_0 (r_0-2m) a_\ml'(r_0) \Big) \\
& a'_\ml(r_0) = \frac{1}{2(r_0-2m) }\left(\ell(\ell+1)-\frac{2m}{r_0}\right)  a_\ml(r_0)
\end{cases}.
\end{align}

\begin{proposition}\label{pr:tilde}
For each $k, \ell$ ranging as $\ell=0,1,2,...$ and $\m=0,\pm 1,...,\pm \ell$, if $ a_{\ml}$ satisfies \eqref{ode-333} and the asymptotics $|a_{\ml}(r) | + |a_{\ml}'(r)| \to 0$  as $r\to \infty$, then $a_{\ml}$ must be identically zero. 
\end{proposition}

When $m=0$, a general solution to the above ODE is given by  $a_{\ml}(r) = c_1 r^{-\ell -1} + c_2 r^\ell$ for some constants $c_1, c_2$. It is direct to verify that the initial value condition implies $c_1=c_2=0$. Thus, for the remainder of the section, we assume $m\neq 0$.  
We begin with an ODE comparison statement.

\begin{lemma}\label{le:ode}
Let $h(r)$ and $p(r)$ be two positive functions defined for $r\ge r_0\ge 0$.  Suppose $B(r)$ solve
\[
	\left( h(r) B' \right) ' =p (r) B\quad \mbox{ on } [r_0, \infty).
\]
\begin{enumerate}
\item Suppose $B(r_0)\ge 0$ and $B'(r_0)\ge 0$ and not both zero. Then $B(r)>0$ and $B'(r)>0$ for all $r> r_0$. \label{it:positive}
\item Suppose $h (r) =r(r-c)$ for some constant $c<r_0$ and $p(r)$ converges to a positive constant as $r\to \infty$. Suppose $B(r)>0$. If $B(r)$ is increasing, then $\lim_{r\to \infty} B(r) = +\infty$. If $B(r)$ is decreasing, then $\lim_{r\to \infty} B(r) = 0$. \label{it:limit}
\end{enumerate}
\end{lemma}
\begin{proof}
To prove the first item, it suffices to prove that $B'(r)>0$ for $r>r_0$. We first assume  $B(r_0)> 0$ and $B'(r_0)> 0$. Suppose, to get a contradiction, there exists $s>r_0$ such that $B'(s)=0$ for the first time. Since we have $B(r) >0$ for all $r_0<r\le s$, $\left( h(r) B' \right) ' =p(r) B> 0$, which implies $\alpha(s) B'(s)>\alpha(r_0) B'(r_0)>0$. It contradicts that $B'(s)=0$. For the case where either $B(r_0)=0$ or $B'(r_0)=0$,  it is straightforward to show that both become strictly positive immediately.

For the second item, since $B(r)>0$ is monotone, either $\lim_{r\to \infty} B(r) = +\infty$ or $\lim_{r\to \infty} B(r) = C$ for some constant $C\ge 0$. Suppose the latter case (which can only occur if $B$ is decreasing). We have $B'(r)$ is integrable.  We show that $rB'(r)$ is bounded: integrating the ODE, we have 
\[
r(r-c) B'(r) = r_0 (r_0-c) B'(r_0) + \int_{r_0}^r \beta(s) B(s) \, ds,
\]
and note that the right hand side is bounded by a constant multiple of $r$. Next, we compute 
\begin{align*}
	rB'(r) - r_0 B'(r_0) &= \int_{r_0}^r \left( s B'(s) \right)' \,ds\\
	&= \int_{r_0}^r \frac{1}{s-c} p(s) B(s) \, ds -\int_{r_0}^r \frac{s}{s-c} B'(s)\, ds.
\end{align*}
Since $rB'(r)$ is bounded and $B'$ is integrable, the first integral in the right hand side must converge to a bounded number as $r\to \infty$ (note the integrand is positive). It implies that  $\frac{1}{s-c} B(s)$ must be integrable and thus $C=0$. 
\end{proof}

We recall the equations for $a_{\ml}$ from \eqref{ode-333} and  omit the subscripts $k, \ell$ for $a_{\ml}$,  as we will work for each fixed $k, \ell$. We also recall $m\neq 0$ for the following arguments. Additionally,  for each fixed $k, \ell$, we denote the constant
\begin{align*}
\alpha_0 &= \tfrac{1}{2m} \left( (4m-r_0) a(r_0) +r_0 (r_0-2m) a'(r_0)\right) \\
&= \left( \tfrac{3}{2} + \tfrac{r_0}{4m} ( \ell (\ell+1) -2)\right) a(r_0). 
\end{align*}

Therefore,  \eqref{ode-333}  becomes
\begin{align}\label{eq:ode}
\begin{cases} 
	&\left(r\left ( r- 2m \right) a' \right)' = \left(\tfrac{4m^2}{ r(r-2m) } + \ell (\ell + 1) \right) a - \tfrac{4 m^2}{r(r-2m) } \alpha_0 \\
	&a'(r_0)= \tfrac{1}{2(r_0-2m)} \left( \ell(\ell+1) - \tfrac{2m}{r_0} \right) a(r_0).
\end{cases}
 \end{align}
 
We will rewrite $a(r)$ and work with the corresponding differential equation in two different ways. Let  $A(r) = a(r) - \alpha_0$. Then $A(r)$ satisfies 
	\begin{align}\label{eq:A}
	\left\{ \begin{array}{l}	\left(r\left ( r- 2m \right) A' \right)'=\left(\tfrac{4m^2}{ r(r-2m) } + \ell (\ell + 1) \right) A + \alpha_0 \ell(\ell+1) \\
		A(r_0)= \left( -\tfrac{1}{2} - \tfrac{r_0}{4m} (\ell(\ell+1)-2) \right) a(r_0)\\
		A'(r_0)= \tfrac{1}{2(r_0-2m)} \left( \ell(\ell+1) - \tfrac{2m}{r_0} \right) a(r_0)
	\end{array}\right..
\end{align}
Alternatively, if $\ell\neq 1$, we let $B(r) = a(r)- \frac{\beta_0}{r(r-2m)} $ where 
\[
\beta_0 = \tfrac{4m^2\alpha_0}{\ell(\ell+1)-2}.
\]
Then $B(r)$ satisfies
	\begin{align}\label{eq:B}
	\left\{ \begin{array}{l}\left( r(r-2m) B'\right)'  = \left(\tfrac{4m^2}{ r(r- 2m ) } + \ell (\ell + 1) \right) B\\
		B(r_0)= a(r_0) - \frac{1}{r_0(r_0- 2m )}\beta_0\\
		B'(r_0)= a'(r_0) + \frac{2(r_0-m)} {(r_0(r_0-2m))^2} \beta_0
	\end{array}\right..
\end{align}
(In fact, one can introduce a change of variables to rewrite the above equation as an associated Legendre equation and analyze the explicit solutions; however, we will not include that approach here.) Again, we should have denoted $A_{\ml}(r)$ and $B_{\ml}(r)$ but since we will work with a fixed $k, \ell$,  the subscript can be omitted for clarity.

 
    

 To prove Proposition~\ref{pr:tilde}, we can assume $a(r_0)\neq  0$. (If $a(r_0)=0$, $a(r)$ must be identically zero by uniqueness.) The proposition is directly implied by following  lemma.

 \begin{lemma}
Assume $A$ solves \eqref{eq:A} and  $B$ solves \eqref{eq:B} and $a(r_0)\neq 0$.
\begin{enumerate}
\item For $\ell=0$, $\lim_{r\to \infty} A(r) = \left(\frac{r_0}{2m} -\tfrac{1}{2} \right) a(r_0)$. (Note that  $\alpha_0 = \left(\tfrac{3}{2} - \frac{r_0}{2m} \right) a(r_0)$ when $\ell=0$, and thus $\lim_{r\to \infty} A(r) \neq -\alpha_0$.)
\item For $\ell = 1$, $A(r)$ cannot converge to a constant, assuming $A'(r)$ converges to zero as $r\to \infty$. 
\item For $\ell\ge 2$, we have either $\lim_{r\to \infty} B(r) = +\infty$, $\lim_{r\to \infty} B(r) = -\infty$, or $A(r)$ cannot converge to a constant, assuming $A'(r)$ converges to zero as $r\to \infty$.
\end{enumerate}
 \end{lemma}
 \begin{proof}

 Let $\phi(r) = r(r-2m) A(r) $. We compute
\begin{align}
	\phi'(r) &= 2 (r-m) A + r(r-2m) A' \label{eq:phi1}\\
	\begin{split}\label{eq:phi2}
	\phi''(r)&=2A + 2 (r-m) A' +\left(\tfrac{4m^2}{r(r-2m) } + \ell (\ell + 1) \right) A + \alpha_0 \ell(\ell+1)\\
	& =  \tfrac{2 (r-m)}{r(r-2m) }\phi'  + \tfrac{\ell(\ell+1)-2}{r(r-2m)}\phi + \alpha_0 \ell(\ell+1)
	\end{split}
\end{align}
where we use \eqref{eq:A}. The initial values are given by
\begin{align*}
	\phi(r_0) &= -r_0 (r_0-2m)\left( \tfrac{1}{2} + \tfrac{r_0}{4m} (\ell(\ell+1)-2)\right) a(r_0) \\
	\phi'(r_0) &= 2(r_0-m) A(r_0) + r_0 (r_0-2m) A'(r_0)\\
	&=- \tfrac{r_0(r_0-2m)}{2m} (\ell (\ell+1)-2) a(r_0)
\end{align*}

We first discuss the case $\ell=0$. In this case, $\phi$ solves
\[
r(r-2m) \phi''(r) = 2(r-m) \phi'(r) -2 \phi.
\]
It is straightforward to verify that a general solution to the above ODE takes the form $\phi(r) = c_0 r^2 + c_1 r - c_1 m $  for some constants $c_0, c_1$. By the initial value, we have $c_0 =\frac{r_0-m}{2m} a(r_0)$ and $c_1 = -r_0 a(r_0)$, and thus 
\[
	A(r) = \tfrac{c_0 r^2 + c_1 r - mc_1}{r(r-2m)} \to c_0= \tfrac{r_0-m}{2m} a(r_0)\quad \mbox{ as } r\to \infty.
\]

Next, we discuss the case $\ell =1$ and assume $a(r_0)>0$, thus $\alpha_0 >0$. Define 
\begin{align} \label{eq:F}
\Phi(r) =\tfrac{1}{r(r-2m)} \phi'(r) = \tfrac{2(r-m)}{r(r-2m)} A(r) + A'(r)
\end{align}
and compute 
\begin{align}
	\Phi'(r)& =\tfrac{1}{r(r-2m)} \phi''(r) - \tfrac{2(r-m)}{\left(r(r-2m)\right)^2} \phi'(r)\notag\\
	&= \tfrac{1}{r(r-2m)} \big( (\ell(\ell+1)-2) A (r)+ \alpha_0 \ell (\ell+1)\big) \label{eq:phi}\\
 &= \tfrac{\alpha_0}{r(r-2m)}  \quad \mbox{ letting $\ell=1$}.\notag
\end{align}
When $\ell=1$,  $\Phi(r_0)=\frac{1}{r_0(r_0-2m)} \phi'(r_0)=0$ and $\Phi'(r)>0$ because $\alpha_0>0$, so $\Phi$ increases to a positive constant,  using that $\Phi'(r)$ is integrable. Thus, the desired conclusion follows  the relation \eqref{eq:F}.

For  $\ell\ge 2$, we will work with $B(r)$. Similarly to $\phi(r)$ defined above, we  define $f(r) =  r ( r-2m )B(r)$. 
Then  as the above equations for $\phi'$ and $\phi''$, we have 
\begin{align}
    f'(r) &= 2 (r-m) B(r) +r(r-2m)B' (r) \notag\\
       f''(r)& = \tfrac{2 (r-m)}{r(r-2m) }f'  + \tfrac{\ell(\ell+1)-2}{r(r-2m)}f. \label{eq:f}
\end{align}
We may assume $a(r_0)>0$, and hence $a'(r_0)>0$,  $f'(r_0) >0$. Note that $B(r_0)$ may change sign.
\medskip

\noindent {\bf Case 1:} We consider the case $B(r_0)\ge 0$ and will show that $B(r)$ diverges to either $+\infty$ or $-\infty$.  We first show that $f(r)>0$ and $B(r)>0$ for all $r>r_0$. Note that $B$ instantly becomes strictly positive because $f'(r_0)>0$. Suppose, on the contrary, there exists the first $s>r_0$ such that either $f'(s)=0$ or $B(s)=0$ (or both). Then $f'(r)>0$ for $r\in [r_0, s)$, and it implies that $B(s)>0$. Then \eqref{eq:f}  implies that $f$ has a local minimum at $s$, contradicting $f'(r)>0$ for $r<s$. 
Next, we show that have $B'\ge 0$ somewhere, and thus Lemma~\ref{le:ode} implies $B(r)$ diverges to $+\infty$.  Suppose, on the contrary, $B'(r)<0$ for all $r$. As $\Phi$ above, we define $F(r) =\frac{1}{r(r-2m)} f'(r) =  \frac{2 (r-m)}{r(r-2m)} B(r) + B'(r)$ and compute
\[
    F'(r) =\frac{1}{r(r-2m)} (\ell(\ell+1)-2) B(r).
 \]
By positivity of $B$,  we have $B'(r)\to 0$ and $B(r)$ converges to a constant, which implies $F(r)\to 0$. Since $F$ is monotonically increasing with $F(r_0) = \tfrac{1}{r_0(r_0-2m)} f'(r_0)>0$, it gives a contradiction.

\medskip
\noindent {\bf Case 2:}  Now consider the case $B(r_0)<0$, and it suffices to assume $B(r)<0$ and $f(r)>0$ for all $r$. If $B$ later becomes nonnegative,  then we have $B(s)\ge 0$ and $B'(s)\ge 0$ for some $s$ (as $B$ must increase from an initial negative value to a nonnegative value somewhere along  the way), with at least one of $B(s)$ and $B'(s)$ being nonzero, which implies $B$ diverges to $+\infty$ by Lemma~\ref{le:ode}. Similarly, if $f$ later becomes zero at~$s$, we have $B(s)=0$ and $B'(s)< 0$, which implies $B$ diverges to $-\infty$ by Lemma~\ref{le:ode}.

We now assume $B(r)<0$ and $f(r)>0$ for all $r$, and will show that $A$ cannot converge to a constant if $A'$ converges to zero. Using $f>0$, we have $B'(r) > -\frac{2(r-m)}{r (r- 2m)} B$ and thus we solve it to obtain 
\[
    B(r) \ge \tfrac{r_0(r_0-2m)}{r(r-2m)} B(r_0).
\]
Therefore, by the initial condition in \eqref{eq:B}, we have 
\[
    a(r) = B(r) + \tfrac{\beta_0}{r(r-2m)} \ge \tfrac{r_0(r_0-2m)}{r (r-2m)} \left( B(r_0)+\tfrac{ \beta_0}{r_0(r_0-2m)}\right) =  \tfrac{r_0(r_0-2m)}{r (r-2m)}  a(r_0) \ge 0.
\]
Then \eqref{eq:phi} implies
\[
    \Phi'(r) \ge  \tfrac{2\alpha_0}{r(r-2m)}.
\]
Integrating the above inequality gives
\[
    \lim_{r\to \infty} \Phi(r) \ge \Phi(r_0)  + \int_{r_0}^\infty \tfrac{2\alpha_0}{s(s-2m)} \, ds.
\]
We evaluate the right hand side:
\begin{align*}
    \Phi(r_0) = \tfrac{1}{r_0(r_0-2m)} \phi'(r_0) =-\tfrac{1}{2m} (\ell(\ell+1)-2) a(r_0)
\end{align*}
and 
\begin{align*}
    \int_{r_0}^\infty \tfrac{2\alpha_0}{s(s-2m)} \, ds &=\left. \tfrac{\alpha_0}{m} \left( \ln \left(r-2m \right) - \ln r \right)\right|_{r=r_0}^{r=\infty} = -\tfrac{ \alpha_0 }{m} \ln \left(1-\tfrac{2m}{r_0} \right)\\
    &> \tfrac{2}{r_0} \alpha_0= \left( \tfrac{3}{r_0} + \tfrac{1}{2m} ( \ell (\ell+1) -2)\right) a(r_0)
\end{align*}
where we use $-\frac{1}{x} \ln (1-x) > 1$ for any $|x|<1$ in the second line. Combining the above inequalities, we obtain $\lim_{r\to \infty} \Phi(r) >\frac{3}{r_0} a(r_0)$. Thus,  $A$ cannot converge to a constant, assuming $A'$ converges to zero.

 \end{proof}

\bibliographystyle{amsplain}
\bibliography{reference}
\end{document}